\documentclass[11pt]{amsart}

\usepackage{amssymb}
\usepackage[all]{xy}
\usepackage[colorlinks=true, urlcolor=rltblue, citecolor=drkgreen, linkcolor=drkred] {hyperref}
\usepackage{color}
\usepackage{pdfsync}
\definecolor{rltblue}{rgb}{0,0,0.4}
\definecolor{drkred}{rgb}{0.6,0,0}
\definecolor{drkgreen}{rgb}{0,0.4,0}
\usepackage{graphicx}
\usepackage[mathscr]{eucal}

\usepackage{lmodern}
\usepackage[capitalize]{cleveref}
\usepackage{amssymb,amsmath,amsthm}
\usepackage{mathtools, MnSymbol}
\usepackage{setspace}
\usepackage{tikz-cd}
\usetikzlibrary{decorations.pathreplacing}
\usepackage[T1]{fontenc}
\usepackage[utf8]{inputenc}
\usepackage{textcomp} 
\usepackage{stmaryrd}
\usepackage{datetime}
\usepackage{todonotes}
\usepackage{xcolor}
\usepackage{mdframed}
\usepackage[all]{xy}

\usepackage{thmtools}
\usepackage{enumitem}

\declaretheorem[numberwithin=section]{theorem}
\declaretheorem[sibling=theorem]{lemma}
\declaretheorem[sibling=theorem]{proposition}
\declaretheorem[sibling=theorem]{corollary}

\theoremstyle{definition}
\declaretheorem[sibling=theorem]{definition}
\declaretheorem[sibling=theorem]{question}

\newcommand{\bigwwedge}{%
  \mathop{
    \mathchoice{\bigwedge\mkern-15mu\bigwedge}
               {\bigwedge\mkern-12.5mu\bigwedge}
               {\bigwedge\mkern-12.5mu\bigwedge}
               {\bigwedge\mkern-11mu\bigwedge}
    }
}

\newcommand{\bigvvee}{%
  \mathop{
    \mathchoice{\bigvee\mkern-15mu\bigvee}
               {\bigvee\mkern-12.5mu\bigvee}
               {\bigvee\mkern-12.5mu\bigvee}
               {\bigvee\mkern-11mu\bigvee}
    }
}

\def\A{\mathcal A}

\def\L{\mathcal L}
\def\M{{\mathcal M}}

\def\N{{\mathcal N}}

\title{Enumerative Combinatorics of homogeneous linear orderings}

\author[Gonzalez]{David Gonzalez}
\address{Department of Mathematics, University of Notre Dame, Notre Dame, IN 46556}
\email{dgonza42@nd.edu}

\thanks{We would like to thank Corrine Yap for her helpful pointers regarding enumerative combinatorics. The author's research at the Hausdorff Institute for Mathematics, Bonn in Fall 2025, was partially funded by the Deutsche Forschungsgemeinschaft (DFG, German Research Foundation) under Germany's Excellence Strategy – EXC-2047/1–390685813.  The author was also partially supported by NSF grant DMS-2401437.}

\date{\today}

\begin{document}

\maketitle
\begin{abstract} 
    We count the number of countable homogeneous colored linear orderings in $k$ colors.
    Relatedly, we count the number of countable $C_{n,m}$-homogeneous linear orderings.
    $C_{n,m}$-homogeneity is a strong homogeneity notion that approximates $sp-$homogeneity, a notion recently uncovered in \cite{CCGHN} to have important computability theoretic properties.
    Explicit formulas are derived for both of the quantities in question, along with asymptotic bounds.
    The objects being counted are generally infinite, and it is not obvious that there are even only finitely many.
    This fact, along with the more precise counting, is demonstrated by corresponding the linear orderings with finite objects.
\end{abstract}

\section{Introduction}

Given a list of axioms, mathematical logic provides several ways of determining what the simplest structure satisfying those axioms is.
Model theory provides the notion of \textit{homogeneity}, first introduced by Fraïssé \cite{Fr86}.
A structure $\M$ is homogeneous if every isomorphism between two finitely generated substructures of $\M$ can be extended to an automorphism of $\M$.
The study of homogeneous structures is vast and intersects combinatorics, descriptive set theory, permutation group theory, topological dynamics, and representation theory (see \cite{Mac11} for a more detailed summary).

For example, the complete graph on countably many vertices $K_\infty$ is homogeneous.
Any two isomorphic finite subgraphs are automorphic (indeed, any permutation of the domain is an automorphism).
There are also more interesting homogeneous graphs, like the random graph.
In fact, there are infinitely many homogeneous graphs (see \cite{Hen71}).
Some structures are always simple, like algebraically closed fields, which are always homogeneous.
Other interesting objects only have finitely many countable homogeneous models.
For example, there are exactly three countable homogeneous tournaments and Boolean algebras, and there is exactly one homogeneous divisible, torsion-free, Abelian group \cite{Lac84}.
We will later note the well-known fact that there are only three countable homogeneous linear orderings: the empty ordering, the orderings of size one, and $\eta$, the unique dense linear orderings without endpoints.
The main objective of this paper is to study the number of countable homogeneous structures for particular natural theories.
We examine colored linear orderings and $C_{n,m}$-linear orderings (which we define and motivate below).

Homogeneous structures are of interest to computability theorists.
This is because homogeneity interacts with a priori unrelated notions of simplicity offered in computable structure theory.
A structure is \textit{relatively computably categorical} if any two presentations of that structure can compute an isomorphism between them.
This notion is expanded to the hierarchy of relatively $\Delta_\alpha$ categoricity (relative computable categoricity is $\alpha=1$), which measures how much additional power is required to compute isomorphism between copies of a structure by counting the number of Turing jumps required.
These notions are intimately tied with \textit{the Scott rank} of a structure, a well-studied robust notion of descriptive complexity connected to many areas of mathematical logic (see \cite{HT22} for an overview).
Adams and Cenzer \cite{AC17} observed that homogeneous structures are always $\Delta_2$-categorical and even computably categorical if they are defined over a finite relational language.
One use of this tool, observed in \cite{CCGHN}, to study the hierarchy of relative computable categoricity, is to add definable functions or predicates into the language of a structure and study the homogeneous models over this expanded language.
This brings us to the definition of $sp-$homogeneity, an object of study in the paper at hand.

\begin{definition}
    Given a linear ordering $\L$, the function $s:\L\to\L$ is defined by $s(x)$ being the successor of $x$ if it exists and equal to $x$ otherwise.
    Dually, $p:\L\to\L$ is defined by $p(x)$ being the predecessor of $x$ if it exists and equal to $x$ otherwise.
    A linear ordering is $sp-$homogeneous if it becomes homogeneous after the functions $s$ and $p$ are added to the language of linear orderings.
\end{definition}

The successor and predecessor on linear orderings capture the relatively intrinsically computably enumerable information about orderings in much the same way that dependence captures that for vector spaces.
$sp-$homogeneous linear orderings are always relatively $\Delta_4$ categorical, and in \cite{CCGHN}, a structural criterion is given for each possible relative categoricity level of $sp-$homogeneous orderings.
Our further study into $sp-$homogeneity is motivated by their apparent usefulness in computability theory.
That said, our main results do not concern $sp-$homogeneity directly, even if they are motivated by its study.
Rather, we look at related, but equally natural, homogeneity notions.
Namely, we develop a key tool used to understand $sp$-homogeneity by understanding the transformation of $sp$-homogeneous orderings into different forms.
We establish close connections between $sp-$homogeneity and a class of relational homogeneity notions that approximate $sp-$homogeneity.
It is natural to break down a functional language into a relational one because of Adams and Cenzer's observation that (finite) relational languages enjoy more attractive computability theoretic properties.
\begin{definition}
    For $n,m\in\omega\cup\{\infty\}$, $\L$ is $C_{n,m}$-homogeneous if when $\L$ is expanded to included definitions for, $\{S_i\}_{i<n}$ (a unary relation stating the existence of $i$ successors), $\{P_j\}_{j<m}$ (a unary relation stating the existence of $i$ predecessors), $\{Adj_k\}_{k<n+m}$ (a binary relation stating that two points are exactly distance $k$ apart) it becomes homogeneous.
\end{definition}
We show that $C_{\infty,\infty}$-homogeneity recovers the concept of $sp-$homogeneity.

\begin{theorem}
    A linear ordering is $C_{\infty,\infty}$-homogeneous if and only if it is $sp-$homogeneous.
\end{theorem}

Furthermore, the finite approximations to $sp-$homogeneity given by $C_{n,m}$-homogeneity for finite $n$ and $m$ are quite interesting themselves.
For example, as homogeneity notions over finite relational languages, it is direct to see that $C_{n,m}$-homogeneous orderings are always relatively $\Delta_3$ categorical.
More directly related to the study in this article, these concepts only depend on the value of $k=n+m+1$, and there are only finitely many $C_{n,m}$-homogeneous linear orderings at each level.
The number of such linear orderings is closely tied to the number of homogeneous linear orderings with $k$ colors, another quantity that we observe is finite.
In particular, we describe a way to injectively produce a homogeneous linear ordering in $k$ colors from a $C_{n,m}$-homogeneous ordering with $k=n+m+1$.
We give explicit formulas counting both the number of $C_{n,m}$-homogeneous linear orderings in terms of $k=n+m+1$ and the number of homogeneous linear orderings with $k$ colors, along with associated asymptotic analyses for these formulas.

\begin{theorem}
    Let $I(k)$ be the number of $C_{n,m}$-homogeneous linear orderings with $k=n+m+1$.
    \[I(k)=\sum_{m=1}^k{k \choose m}\sum_{n=1}^m\sum_{r=0}^{\lceil \frac{n}2 \rceil}{n-r+1 \choose r}{m \choose r}r!(n-r)!S(m-r,n-r),\]
    Where $S(n,m)$ are the Stirling numbers of the second kind.
    Furthermore, $I(k)=O(k!2.123^k)$.

    Let $L(k)$ be the number of homogeneous linear orderings with $k$ colors.
    $L(k)$ are the coefficients in the exponential generating function
    \[H(x)=\frac{e^x}{2-x-e^x}.\]
    Furthermore,
    \[L(k)\sim -k!R\Big(\frac1Z\Big)^{k+1},\]
    where $Z=2-W(e^2)\approx0.442854$ and $R=\frac{-e^2}{e^{W(e^2)}+e^2}\approx-0.6089389.$
\end{theorem}

This article has four sections, including the current one that contains introductory material.
The second section gives background material on $sp-$homogeneity and colored linear orderings.
The third section establishes a bijection between the $sp-$homogeneous orderings, colored linear orderings, and an independently defined theory of multicolored orderings.
The last section uses the bijection from the third section to show that there are finitely many $C_{n,m}$-homogeneous orderings, finitely many homogeneous $k$-colored linear orderings, and determine these finite quantities.

\section{Background}

Several constructions of linear orderings are quite important in this work.
We cover all that is needed here.
Generally speaking, we follow the conventions in \cite{Ros}.

$\omega$ is used to denote the order type of the natural numbers.
$\zeta$ is used to denote the order type of the integers.
Given a linear ordering $\L$, we use $\L^*$ to denote the reverse ordering of $\L$, i.e., $x<_\L y$ if and only if $y<_{\L^*}x$.
In context, the natural number $n$ will be used to denote the unique linear ordering with $n$ elements.
Similarly, as appropriate, $\mathbb{N}$ will denote the set of finite linear orderings.

The relation $\sim_1$ can be defined on any linear ordering.
$x\sim_1y$ in $\L$ if there are only finitely many elements between $x$ and $y$ in $\L$.
$\sim_1$ is an equivalence relation.
Furthermore, it is convex, or, what is the same, if $x<y<z$ and $x\sim_1z$ then $x\sim_1y\sim_1z$.
The equivalence classes of $\sim_1$ are called 1-blocks.
Given $x\in\L$ we let $[x]\in \L/\sim_1$ denote the 1-block that $x$ lies in.
Because $\sim_1$ is convex, there is an induced linear ordering on $\L/\sim_1$ given by $[x]<[y]$ if for any $x'\in[x]$ and $y'\in[y]$ we have $x'<y'$.
Each 1-block is itself a linear ordering.
1-blocks have the special property that any two elements are only finitely far apart.
For this reason, up to isomorphism, every 1-block is of the form $\omega,\omega^*,\zeta$ or some finite $n$.

As we are studying linear orderings with homogeneity is properties, it is also important that we define the shuffle sum.
Because we also deal with colored linear orderings here, we also define the color shuffle.
A colored linear ordering is a linear ordering partitioned by a family of unary predicates called colors.
We will use the notation $LO_\infty$ to refer to the set of countable linear orderings with countably many colors.
Similarly, $LO$ is the set of countable linear orderings and $LO_k$ is the set of countable linear orderings with $k$ colors.
\begin{definition}
    Given a countable set of colors $A$, the color shuffle of $A$, denoted $Sh_c(A)\in LO_\infty$, is given by the unique countable colored linear ordering where each color in $A$ is dense.
    Given a countable set of linear orderings $S$ the shuffle sum of $S$, denoted $Sh(S)\in LO$, is given by bijectively associating $f:A\to S$ for some set of colors $A$ and replacing each point in $Sh_c(A)$ of color $c\in A$ with a copy of the linear ordering $f(c)$.
\end{definition}
It is a standard fact that the color shuffle and shuffle sum of a countable set of linear orderings is well-defined up to isomorphism.
This is confirmed by a back-and-forth argument extending the uniqueness of the dense linear ordering without endpoints (see \cite[Chapter 7.3]{Ros}).

Shuffle sums are critical in the classification of $sp-$homogeneous linear orderings from \cite{CCGHN}.

\begin{theorem}\label{thm:homogChar}  Let $L$ be a linear ordering.
$L$ is $sp$-homogeneous if and only if there is a pairwise disjoint family $\{A,A_v, v \in V\}$ of subsets of $\mathbb N \cup \{\omega,\omega^*,\zeta\}$ such that $L$ is the union of suborderings, as follows: 

\begin{enumerate}
\item For each $v \in V$, an open interval $I_v$ isomorphic to $Sh(A_v)$;

\item For each $a \in A$, a single block of size $a$.
\end{enumerate}

\end{theorem} 

This is remarkably similar to the classification of homogeneous, colored linear orderings seen in Lemma 3.1 of \cite{ST08}.
We will later see that this similarity is not a coincidence.
This Lemma has been rewritten below in a way that obviously echoes the above theorem.

\begin{lemma}
Let $L$ be a colored linear ordering with color set $C$.
$L$ is homogeneous if and only if there is a pairwise disjoint family $\{A,A_v, v \in V\}$ of subsets of $C$ such that $L$ is the union of suborderings, as follows: 

\begin{enumerate}
\item For each $v \in V$, an open interval $I_v$ isomorphic to $Sh_c(A_v)$;

\item For each $a \in A$, a single element of color $a$.
\end{enumerate}
\end{lemma}

When the color set $C$ is finite, we will count the number of objects in this form.
The homogeneous structures in $LO_k$ will roughly correspond with $C_{n,m}-$homogeneous linear orderings with $k=n+m+1$.

\section{Colorings and Predicates}\label{sec:colorings}
The purpose of this section is to establish a correspondence between $sp$-homogeneous linear orderings, certain types of homogeneous colored linear orderings, and certain types of linear orderings with predicates.
We do this for two reasons.
First, it aims to give an explicit method of constructing $sp$-homogeneous linear orderings and homogeneous colored linear orderings through this correspondence.
To be specific, it will be very clear how to construct these linear orderings with predicates, and the methods of this section will allow one to readily transform that into an $sp$-homogeneous linear ordering or homogeneous colored linear ordering.
The second purpose is that under the maps defined in this section, some $sp$-homogeneous linear orderings and colored linear orderings will be sent to finite objects.
This will allow us to establish precise counting results for these orderings utilizing this correspondence in the next section.

\begin{definition}
Let $T$ be the theory that is defined by the following axioms in the language $(<,w,z,ws,\{c_n\}_{n\geq 1})$. All of the $w,z,ws$ and the $c_n$ are unary predicates; when referring to them collectively, we sometimes (as in items (2) and (3) below) write $w=c_{-2}$,$z=c_{-1}$, and $sw=c_{0}$.
\begin{enumerate}
	\item $<$ is a linear ordering
	\item $\forall x ~ \bigvvee_{n\geq -2} c_n(x)$
	\item $\forall x \bigwwedge_{n\neq m\geq -2} \lnot\big(c_n(x) \land c_m(x)\big)$
	\item For each $i,j\geq 0$, $\forall x<y ~ \big(c_i(x) \land c_j(y)\big)\to \exists z ~ x<z<y$
	\item For each $i\geq 0$, $\forall x<y ~ \big(c_i(x) \land w(y)\big)\to \exists z ~ x<z<y$.
	\item For each $i\geq 0$, $\forall x<y ~ \big(sw(x) \land c_i(y)\big)\to \exists z ~ x<z<y$.
    \item $\forall x<y ~ \big(sw(x) \land w(y)\big)\to \exists z ~ x<z<y$
\end{enumerate}
\end{definition}

$T$ is a theory that extends the theory of linear orderings with countably many colors.
In addition to stating that a model is a linear ordering with countably many colors, it enforces that many combinations of colors cannot exist next to each other without another element in between.
Let $S$ be the theory of $\{<,s,p\}$ orderings.
\begin{proposition}\label{prop:coloriffSp}
There are inverse embeddings $\Phi:Mod(S)\to Mod(T)$ and $\Psi:Mod(T)\to Mod(S)$.
\end{proposition}

\begin{proof}
We let the underlying ordering of $\Phi(L)$ be $L/\sim_1$.
$w([x])$ holds if $[x]\cong\omega$.
$z([x])$ holds if $[x]\cong\zeta$.
$sw([x])$ holds if $[x]\cong\omega^s$.
$c_n([x])$ holds if $[x]\cong n$.
It is clear that $\Phi(L)$ is a colored linear ordering; in other words, it satisfies axioms (1)-(3) of $T$.
We now need to check axioms (4)-(7).
In each case, these axioms state that particular colors cannot be adjacent.
These correspond to saying that blocks of particular order types are not adjacent to each other in $L/\sim_1$.
Axiom (4) holds so long as no finite blocks are adjacent to each other.
Indeed, this is impossible as adjacent finite blocks become one larger finite block.
Axioms (5), (6), and (7) can be similarly checked.
They follow because a finite block cannot precede an $\omega$ block directly, because a finite block can not come directly after an $\omega^*$ block, and because an $\omega$ block cannot come directly after an $\omega^*$ block, respectively.

For $x\in K\models T$, let
\[
D(x) = \begin{cases}
(n,<,s,p) \text{ if } K\models c_n(x) \\
(\omega,<,s,p) \text{ if } K\models w(x) \\
(\omega^*,<,s,p) \text{ if } K\models sw(x) \\
(\zeta,<,s,p) \text{ if } K\models z(x).
\end{cases}
\]
This map is total and well-defined by Axioms (2) and (3) of $T$.
Let
\[\Psi(K)=\sum_{x\in K} D(x). \]
To check that $\Psi(K)$ is always a model of $S$, one needs only to confirm that there are no missing successors or predecessors that are not accounted for within each $D(x)$.
In other words, one must check that there are no adjacent pairs in $\Psi(K)$ that are not within a single $D(x)$.
This can only happen if there are successive elements $x,y$ such that $D(x)$ has a greatest element and $D(y)$ has a least element.
Axioms (4)-(7), along with the definition of $D$, forbid this from happening.

Checking that $\Phi$ and $\Psi$ are two-sided inverses of each other is routine; it follows immediately from the definitions.
\end{proof}

The following proposition demonstrates that the above-defined mappings preserve the critical property of homogeneity.
Therefore, we observe a close tie between $sp-$homogeneous linear orderings and homogeneous colored linear orderings.

\begin{proposition}
$(L,<,s,p)$ is homogeneous if and only if $\Phi(L)$ is homogeneous.
\end{proposition}

\begin{proof}
Say that $L$ is homogeneous.
Take some $\bar{a},\bar{b}\in \Phi(L)$ that have the same quantifier-free diagram.
These tuples correspond to a tuple of 1-blocks in $L$.
Let $\bar{a}'$ select elements from each 1-block in the following manner: it selects the first element if there is one, if not, the last element if there is one, if not, an arbitrary element.
More explicitly, the first element in finite and $\omega$ blocks the last element of an $\omega^*$ block and any element of a $\zeta$ block.
Let $\bar{b}'$ be similarly defined.
Note that, as $\bar{a}$ and $\bar{b}$ are colored the same, the structures generated by $\bar{a}'$ and $\bar{b}'$ are isomorphic.
Therefore, there is an automorphism $\sigma$ of $L$ sending $\bar{a}'$ to $\bar{b}'$.
$\sigma$ descends to the quotient, so $\sigma/\sim_1$ a linear ordering automorpism on $\Phi(L)$ sending $\bar{a}$ to $\bar{b}$.
Furthermore, as $\sigma$ preserves the $s$ and $p$ structure, it must send each block to an isomorphic block.
This means that $\sigma/\sim_1$ preserves the coloring as well, and it is a full automorphism.
Therefore, $\Phi(L)$ is homogeneous.

Say that $K\models T$ is homogeneous.
We show that $\Psi(K)$ is homogeneous, which is enough to prove the theorem.
Take some $\bar{a},\bar{b}\in \Psi(K)$ that generate isomorphic substructures.
Let $\bar{[a]}$ be the set of 1-blocks that $\bar{a}$ intersects and defined $\bar{[b]}$ similarly.
Because $\bar{a}$ and $\bar{b}$ generate the same isomorphism types of blocks, $\bar{[a]}$ and $\bar{[b]}$ must be colored the same $K$.
By homogeneity there is an automorphism $f$ sending $\bar{[a]}$ to $\bar{[b]}$.
Write $\Psi(K)$ as pairs $(x,y)$ with $x\in K$ and $y\in D(x)$ where $D(x)$ is presented in a standard manner.
We define $\hat{f}: \Psi(K)\to\Psi(K)$ as $\hat{f}(x,y)=(f(x),y)$.
This map is well defined because $D(x)\cong D(f(x))$.
It is straightforward to confirm that it is an automorphism.
Therefore, $\Psi(K)$ is homogeneous, as desired.

\end{proof}

The above explains the similarity of Theorem \ref{thm:homogChar} and Lemma 3.1 of \cite{ST08}.

We take our analysis a step beyond colored linear orderings by creating a correspondence of these homogeneous classes with an infinitary axiomatized class, making no reference to the concept of homogeneity.

\begin{definition}
    Let $T'$ be the theory that is defined by the following axioms in the language $(<,\{R_i\}_{i\in\omega},w,z,ws,\{S_i\}_{i\geq-2})$. All of the $R_i$, $w:=R_{-2}$, $z_:=R_{-1}$, $ws:=R_0$ and $S_i$ are unary predicates.
    \begin{enumerate}
        \item $<$ is a linear ordering.
        \item $\forall x ~ \bigvvee_{i>-2} R_i(x)\lor S_i(x)$
	\item The $R_i$ points act as a model of $T$, i.e.,
        \begin{enumerate}
            \item  $\forall x \bigwwedge_{n\neq m\geq -2} \lnot\big(R_n(x) \land R_m(x)\big)$;
            \item For each $i,j\geq 0$, $\forall x<y ~ \big(R_i(x) \land R_j(y)\big)\to \exists z ~ x<z<y$;
	    \item For each $i\geq 0$, $\forall x<y ~ \big(R_i(x) \land w(y)\big)\to \exists z ~ x<z<y$;
	    \item For each $i\geq 0$, $\forall x<y ~ \big(sw(x) \land R_i(y)\big)\to \exists z ~ x<z<y$;
            \item $\forall x<y ~ \big(sw(x) \land w(y)\big)\to \exists z ~ x<z<y$.
        \end{enumerate}
        \item $\forall x \bigwwedge_{i,j\in\omega} \lnot\big(R_i(x) \land S_j(x)\big)$.
        \item For each $i\geq-2$, $\forall x,y ~ \big(R_i(x)\land R_i(y)\big)\to x=y$.
        \item For each $i\geq-2$, $\forall x,y ~ \big(S_i(x)\land S_i(y)\big)\to x=y$.
        \item For each $i\geq-2$, $\forall x,y ~ \lnot\big( R_i(x)\land S_i(y)\big)$.
    \end{enumerate}
\end{definition}

Intuitively, one should think of a model of $T'$ as a \textit{multi-colored} linear ordering.
A multi-coloring allows you to assign more than one color to each element.
Think of $R_i$ and $S_i$ as both assigning the color $i$ to an element. The difference between the two is that any point can have as many as infinitely many of the $S_i$ assigned to it, but only one of the $R_i$.
In addition to this, each color can only be used at most one time.
Lastly, the restrictions on successors present in $T$ remain present on the $R_i$ points in $T'$.
It should be easy to imagine how to construct models of $T'$; one only needs to take a linear ordering and provide a multicoloring that satisfies rather simple rules.

As we see below, models of $T'$ exactly enumerate the sp-homogeneous linear orderings.
In particular, we can capture the sp-homogeneous linear orderings as constructions of models of $T'$.
Concretely, this gives a ready procedure to construct the sp-homogeneous linear orderings and provides a global characterization of this class.

\begin{proposition}\label{prop:spImage}
    There is a computable injective map $F:Mod(T')\to Mod(LO)$ such that the image of $F$ is exactly the sp-homogeneous linear orderings. 
\end{proposition}

\begin{proof}
    We define a map $G:Mod(T')\to Mod(T)$ and let $F=\Psi\circ G$.
    Let $\L\models T'$.
    We define $E:\L\to Mod(T)$ as
    \[ E(x) = \begin{cases}
        x \text{ with color } i \text{ if } \L\models R_i(x) \\
        Sh_c(A) \text{ if } \L\models S_i(x) \iff i\in A.
    \end{cases}\]
    Note that $E$ is well-defined by Axioms (2)-(4).
    Let
    \[G(\L)= \sum_{x\in\L}E(x).\]

    We first show that $F(\L)$ is always sp-homogeneous.
    It follows from Theorem \ref{thm:homogChar} that it is enough to show that every element is either in a unique block type or lies in an open interval isomorphic to a shuffle of blocks with unique block isomorphism types.
    Note that if there is some $i$ for which $\L\models R_i(x)$, then $x$ is in a singleton block, whereas if $\L\models S_i(x)$, it lies in some shuffle sum.
    Furthermore, the block sizes are not repeated because of Axioms (3)-(7).
    Altogether, we see that $F(\L)$ is always sp-homogeneous.

    Say that $\A$ is $sp$-homogeneous.
    By Theorem \ref{thm:homogChar}, $\Phi(\A)$ locally is either a colored shuffle (using unique colors) or a single uniquely colored point.
    In particular, $x\sim_c y$ defined by $\exists v,w ~ x\leq v\leq w\leq y$ such that $v$ and $w$ share the same color is a convex equivalence relation.
    $\Phi(\A)/\sim_c$ simply collapses the equivalence colored shuffle sums.
    Each element $[x]\in \Phi(\A)/\sim_c$ is either a singleton of a particular color, or it comes from a shuffle sum.
    In the former case, we may assign $\Phi(\A)/\sim_c\models R_i([x])$ where the singleton was colored $i$.
    In the latter case, we may assign $\Phi(\A)/\sim_c\models S_i([x])$ if and only if $i$ is among the colors that were shuffled in $[x]$.
    It follows routinely that $\Phi(\A)/\sim_c\models T'$ and that $F(\Phi(\A)/\sim_c)=\A$ as required.
\end{proof}

\begin{corollary}
      Let $T''$ be the theory that is defined by the following axioms in the language $(<,\{R_i\}_{i\in\omega},\{S_i\}_{i\in\omega})$. All of the $R_i$ and $S_i$ are unary predicates.
    \begin{enumerate}
        \item $<$ is a linear ordering.
        \item $\forall x ~ \bigvvee_{i\in\omega} R_i(x)\lor S_i(x)$
        \item $\forall x \bigwwedge_{n\neq m\in\omega} \lnot\big(R_n(x) \land R_m(x)\big)$
        \item $\forall x \bigwwedge_{i,j\in\omega} \lnot\big(R_i(x) \land S_j(x)\big)$.
        \item For each $i\in\omega$, $\forall x,y ~ \big(R_i(x)\land R_i(y)\big)\to x=y$.
        \item For each $i\in\omega$, $\forall x,y ~ \big(S_i(x)\land S_i(y)\big)\to x=y$.
        \item For each $i\in\omega$, $\forall x,y ~ \lnot\big( R_i(x)\land S_i(y)\big)$.
    \end{enumerate}
    There is a computable injective map $H:Mod(T'')\to Mod(LO_\infty)$ such that the image of $G$ is exactly the homogeneous colored linear orderings. 
\end{corollary}

\begin{proof}
    We let $H$ be defined exactly as $G$ was in the previous proposition.
    In other words, we define $E:\L\to Mod(LO_\infty)$ as
    \[ E(x) = \begin{cases}
        x \text{ with color } i \text{ if } \L\models R_i(x) \\
        Sh_c(A) \text{ if } \L\models S_i(x) \iff i\in A.
    \end{cases}\]
    Note that $E$ is well-defined by Axioms (2)-(4).
    Let
    \[G(\L)= \sum_{x\in\L}E(x).\]
    The analysis in the previous proposition demonstrates that the injective image of $H$ is the set of colored linear orderings composed of elements with a unique color and shuffle sums of unique sets of colors (though, in this case, not subject to the adjacency requirements of $T'$).
    By \cite{ST08} Lemma 3.1, these are exactly the homogeneous models of $LO_\infty$.
\end{proof}

\section{Relational approximations to $sp$-homogeneity}

In this section, we break down the concept of $sp$-homogeneity, which is defined using functions, in terms of a list of the notions of $C_{n,m}$-homogeneity defined in the introduction.
This gives us a countable list of approximate strengthenings of $sp$-homogeneity that align with finitely colored linear orderings instead of the infinitely colored ones constructed in Section \ref{sec:colorings}.
This allows us to correspond these strengthenings with finite objects and, ultimately, count both the number of models with these properties and the number of homogeneous linear orderings on $k$ colors.
Let us recall the definitions needed for $C_{n,m}$-homogeneity.

\begin{definition}
    Let $P_n$ be a unary predicate that holds of the elements in a linear ordering with $n$ predecessors.
    Let $S_n$ be a unary predicate that holds of the elements in a linear ordering with $n$ successors.
    Let $Adj_n$ be a binary predicate that holds of pairs of elements with exactly $n$ elements in between them.
\end{definition}

\begin{definition}
    For $n,m\in\omega\cup\{\infty\}$ let $C_{n,m}$ be the theory over the language $<$, $\{S_i\}_{i<n}$, $\{P_j\}_{j<m}$, $\{Adj_k\}_{k<n+m}$ where $<$ is a linear ordering and the $S_i$, $P_j$ and $Adj_k$ predicates are defined according to the above definition.
\end{definition}

\begin{definition}
    A linear ordering is $C_{n,m}$-homogeneous if its expansion by definition to include $\{S_i\}_{i<n}$, $\{P_j\}_{j<m}$ and $\{Adj_k\}_{k<n+m}$ is homogeneous.
    A linear ordering is weakly $C_{n,m}$-homogeneous if its expansion by definition to include $\{S_i\}_{i<n}$, $\{P_j\}_{j<m}$ and $\{Adj_k\}_{k<n+m}$ is weakly homogeneous.
\end{definition}

We show that these new notions of homogeneity are indeed stronger than $sp$-homogeneity.

\begin{lemma}\label{lem:homogeneitytransfer}
    For any $n$ and $m$, $C_{n,m}$-homogeneous linear orderings are $sp-$homogeneous.
\end{lemma}

\begin{proof}
    Let $\L$ be $C_{n,m}$-homogeneous.
    Consider two tuples $\bar{p},\bar{q}\in\L$ such that the finitely $sp-$generated structures $(\langle\bar{p}\rangle,\bar{p})\cong(\langle\bar{q}\rangle,\bar{q})$.
    Note that this implies that every element in $\bar{p}$ must have the same number of predecessors and successors in $\langle\bar{p}\rangle$ as $\bar{q}$ does in $\langle\bar{q}\rangle$.
    Furthermore, there must be the same number of elements between corresponding $\bar{p}$ points in $\langle\bar{p}\rangle$ and $\bar{q}$ points in $\langle\bar{q}\rangle$.
    However, the number of predecessors and successors that a tuple has in its $sp-$hull is the same as the number of predecessors and successors that it has in the full structure.
    Similarly, the number of points between elements in the $sp-$hull matches the number of points between elements in the full structure.
    In particular, the same $\{S_i\}_{i<n}$, $\{P_j\}_{j<m}$ and $\{Adj_k\}_{k< n+m}$ hold of all of the elements in $\bar{p}$ and $\bar{q}$.
    By $C_{n,m}$-homogeneity, these tuples are automorphic.
    Thus, $\L$ is sp-homogeneous.
\end{proof}

The following proposition allows us to be more precise about which $sp$-homogeneous linear orderings are actually $C_{n,m}$-homogeneous.

\begin{proposition}\label{prop:blockcriteria}
        For any $n$ and $m$, $\L$ is $C_{n,m}$-homogeneous if and only if $\L$ is sp-homogeneous and the automorphism orbits in each block of $\L$ are quantifier-free definable in $\{S_i\}_{i<n}$, $\{P_j\}_{j<m}$ and $\{Adj_k\}_{k<n+m}$.
\end{proposition}

\begin{proof}
    We begin with the forward direction.
    Say that $\L$ is $C_{n,m}$-homogeneous.
    For the sake of contradiction, say that two non-automorphic tuples $\bar{p}$ and $\bar{q}$ in a block $b$ have the same quantifier-free $C_{n,m}$-diagram in $b$.
    As the quantifier-free $C_{n,m}$-diagram is determined by information local to $b$, this means that these two tuples have the same quantifier-free $C_{n,m}$-diagram in $\L$.
    By $C_{n,m}$-homogeneity, these two tuples are automorphic in $\L$ via the automorphism $\sigma$.
    Because $\sigma(\bar{p})=\bar{q}$ and it is an automorphism, $\sigma$ must restrict to and automorphistm $\sigma:b\to b$.
    This contradicts that $\bar{p}$ and $\bar{q}$ are not automorphic in $b$.
    By Lemma \ref{lem:homogeneitytransfer}, the ordering is also sp-homogeneous, as desired.

    We now demonstrate the backwards direction.
    Say that $\L$ is sp-homogeneous and the automorphism orbits in each block of $\L$ are quantifier-free definable in $\{S_i\}_{i<n}$, $\{P_j\}_{j<m}$ and $\{Adj_k\}_{k<n+m}$.
    Take two tuples $\bar{p}$ and $\bar{q}$ that have the same quantifier-free $C_{n,m}$-diagram.
    By assumption, $\bar{p}$ and $\bar{q}$ are automorphic within their blocks.
    What is the same, $(\langle\bar{p}\rangle,\bar{p})\cong(\langle\bar{q}\rangle,\bar{q})$.
    In particular, by $sp$-homogeneity, they are automorphic.
    Therefore, $\L$ is $C_{n,m}$-homogeneous as desired.
\end{proof}

Putting things together, we obtain the desired result that the notions of $C_{n,m}$-homogeneity approximate $sp$-homogeneity in a precise way.
As $n$ and $m$ tend to infinity, the class of $C_{n,m}$-homogeneous linear orderings approaches the class of $sp$-homogeneous linear orderings in the sense that the language in which $C_{n,m}$-homogeneous linear orderings are homogeneous captures increasing shares of the expressive power of the $sp$ language.

\begin{theorem}\label{thm:CnmClassify}
    \begin{enumerate}
        \item If $n,m<\infty$ then the $C_{n,m}$-homogeneous structures are the sp-homogeneous structures with blocks of size at most $n+m+1$.
        \item If $n=\infty$ and $m<\infty$ then the $C_{n,m}$-homogeneous structures are the sp-homogeneous structures without blocks isomorphic to $\omega$.
        \item If $m=\infty$ and $n<\infty$ then the $C_{n,m}$-homogeneous structures are the sp-homogeneous structures without blocks isomorphic to $\omega^*$.
        \item If $m=n=\infty$ then the $C_{n,m}$-homogeneous structures are exactly the sp-homogeneous structures.
    \end{enumerate}
\end{theorem}

\begin{proof}
    Each of the claims follows from Proposition \ref{prop:blockcriteria}.
    It is direct to check that $\{S_i\}_{i<n}$ and $\{P_j\}_{j<m}$ define every element in a finite block if and only if that block has size at most $n+m+1$.
    Moreover, in a block isomorphic to $\omega$ the set $\{P_j\}_{j<\omega}$ defines all elements, but any $\{S_i\}_{i<n}$ and $\{P_j\}_{j<m}$ with $m$ finite fails to.
    Similarly, in a block isomorphic to $\omega^*$ the set $\{S_j\}_{j<\omega}$ defines all elements, but any $\{S_i\}_{i<n}$ and $\{P_j\}_{j<m}$ with $n$ finite fails to.
    Lastly, blocks isomorphic to $\zeta$ require all of the $Adk_k$ predicates to define the automorphism orbits, so are permissible exactly when $n$ or $m$ is infinite.
\end{proof}

The most intriguing feature of these languages is that they only have finitely many homogeneous models.
This provides an interesting line of inquiry regarding the (finite) combinatorics of these structures that we explore below.

\begin{theorem}
Let $n$ and $m$ be finite.
\begin{enumerate}
    \item There are only finitely many $C_{n,m}$-homogeneous linear orderings.
    \item The number of $C_{n,m}$-homogeneous linear orderings only depends on $k=n+m+1$.
    \item If $I(k)$ is the number of $C_{n,m}$-homogeneous linear orderings with $k=n+m+1$, then $I(k)$ satisfies the following recursive definition.
    Let
    \[K_1(0)=1,K_2(0)=0 \text{ and }\] 
    \[K_1(k+1)= \sum_{i=0}^{k} {k+1 \choose i} \big(K_1(i)+K_2(i)\big) \text{ and } K_2(k+1)=(k+1)K_1(k).\]
    Then
    \[I(k)=\sum_{i=0}^{k} \big(K_1(i)+K_2(i)\big) {k \choose i}.\]
    \item $I(k)$ has the closed form
    \[I(k)=\sum_{m=1}^k{k \choose m}\sum_{n=1}^m\sum_{r=0}^{\lceil \frac{n}2 \rceil}{n-r+1 \choose r}{m \choose r}r!(n-r)!S(m-r,n-r),\]
    Where $S(n,m)$ are the Stirling numbers of the second kind.
    \item $I(k)$ is equal to the number of homogeneous linear orderings in $k$ colors satisfying $T$.
    \item $I(k)$ is equal to the number of models of $T'_k$, where $T'_k$ is the reduction of the theory $T'$ to the language $R_1,\dots,R_k$ and $S_1,\dots,S_k$.
\end{enumerate}
     
\end{theorem}

\begin{proof}
    Note that (2) follows immediately from Theorem \ref{thm:CnmClassify} and that (1) will follow from (3).
    By Proposition \ref{prop:spImage} and Theorem \ref{thm:CnmClassify}, every $C_{n,m}$-homogeneous linear ordering lies in the image $F(T')$.
    Furthermore, following the definition of $F$, any element with predicate outside of $R_1,\dots,R_k$ and $S_1,\dots,S_k$ results in a block of size at least $k+1$ in $F(T')$.
    Conversely, if only $R_1,\dots,R_k$ and $S_1,\dots,S_k$ are used, only blocks of size at most $k$ are produced in $F(T')$.
    This means that $F$ is a bijection between the models of $T'_k$ and the $C_{n,m}$-homogeneous linear orderings.
    The same analysis for $\Psi$ from Proposition \ref{prop:coloriffSp} shows that the homogeneous linear orderings satisfying $T$ in $k$ colors (only the colors $c_1,\dots,c_k)$ are in bijection with the $C_{n,m}$-homogeneous linear orderings.
    In particular, (5) and (6) follow at once.

    All that remains is demonstrating (3) and (4).
    To accomplish this, we will explicitly count the models of $T'_k$, and appeal to the above analysis to observe that this gives the other desired counts as well.
    We begin with (3).
    We interpret $K_1(i)$ as the number of models of $T'_i$ using all $i$ colors where the first element satisfies some $R_i$, and interpret $K_2(i)$ as the number of models of $T'_i$ using all $i$ colors where the first element satisfies some $S_i$.
    Note that $\A\models T'_{k+1}$ must have finitely many elements; no color can be reused, and every element takes at least one color.
    Therefore, there is always a first element.
    If we are able to demonstrate that $K_1(i)$ and $K_2(i)$ act as we described above, the formula
    \[I(k)=\sum_{i=1}^{k} \big(K_1(i)+K_2(i)\big) {k \choose i},\]
    follows at once.
    This is because a model of $T'_k$ uses $i\leq k$ colors.
    It has ${k \choose i}$ options for that set of colors and $K_1(i)+K_2(i)$ many underlying models for that choice of colors.
    We focus on demonstrating that $K_1$ and $K_2$ are as desired.
    We demonstrate this fact inductively.
    Say that this characterization holds for $i\leq k$.
    Consider a model $\A\models T'_{k+1}$ using all of the colors.
    Say that the first element of $\A$ has $k+1-i$ of the $S_j$ hold of it.
    Note that there are ${k+1 \choose i}$ choices for which of the colors are used.
    The remaining ordering is a model of $T_i'$, of which there are $K_1(i)+K_2(i)$ many by induction.
    Therefore,
    \[K_1(k+1)= \sum_{i=0}^{k} {k+1 \choose i} \big(K_1(i)+K_2(i)\big),\]
    as desired.
    Now consider the case that the first element of $\A$ has some $R_j$ hold of it.
    Note that there are $k+1$ choices for which of the colors is used.
    The remaining ordering is a model of $T_k'$, but it necessarily does not begin with an element with some $R_i$ holding (otherwise, two such points would be adjacent, which is forbidden by our axioms).
    There are $K_1(k)$ many of these models by induction.
    Therefore,
    \[K_2(k+1)=(k+1)K_1(k),\]
    as desired.

    We now demonstrate (4) by counting the models of $T_k'$ directly.
    Each term in the sum
    \[I(k)=\sum_{m=1}^k{k \choose m}\sum_{n=1}^m\sum_{r=0}^{\lceil \frac{n}2 \rceil}{n-r+1 \choose r}{m \choose r}r!(n-r)!S(m-r,n-r),\]
    will be interpreted explicitly.
    Interpret $m$ as the total number of colors used in $T_k'$, and ${k \choose m}$ counts the number of choices for said colors.
    Therefore, it is enough to show that there are
    \[\sum_{n=1}^m\sum_{r=0}^{\lceil \frac{n}2 \rceil}{n-r+1 \choose r}{m \choose r}r!(n-r)!S(m-r,n-r)\]
    many models of $T_m'$ using all of the $m$ colors.
    Interpret $n$ as the number of elements in the model.
    Interpret $r$ as the number of elements satisfying some $R_i$ in the model.
    As these elements cannot be successive, there are at most $\lceil \frac{n}2 \rceil$ such elements.
    The isomorphism type of the model is now determined by
    \begin{enumerate}
        \item A choice of which elements are $R_i$ elements
        \item A choice of the colors used on the $R_i$ elements
        \item A choice of the order of the colors on the $R_i$ elements
        \item A choice of the order of the color sets for the $S_i$ elements
        \item A partition of the colors among the $S_i$ elements.
    \end{enumerate}
    Note that each of these choices is made independently.
    Therefore, we can count each quantity and multiply them together to get the final number of models subject to the parameters $k,m,n,r$.
    The counts are as follows:
    \begin{enumerate}
        \item ${n-r+1 \choose r}$, which is a known formula for the number of length $n$ binary sequences with $r$ ones, no two of which are consecutive.
        \item ${m \choose r}$, by definition
        \item $r!$, by definition
        \item $(n-r)!$, by definition
        \item $S(m-r,n-r)$, as we are splitting $m-r$ ordered colors among $n-r$ ordered points, and the Stirling numbers of the second kind count exactly this quantity.
    \end{enumerate}
    This completes the proof.
\end{proof}

The first few terms for $I(k)$ are given as follows: 3, 12, 71, 558, 5487, 64734, 891039, 14016774, 248057927, 4877703126,
105504350679, 2489510252238, 63638447941551,...

As we saw above in (5), the $C_{n,m}$-homogeneous linear orderings are closely related to the homogeneous linear orderings in $k$ colors.
These theories are of classical interest.
We now analyze the closely related quantity of homogeneous colored linear orderings in $k$ colors without the restrictions on successivities.
In many ways, this sequence of numbers is better behaved.

\begin{theorem}
Let $L(k)$ equal the number of homogeneous linear orderings in $k$ colors.
    \begin{enumerate}
    \item $L(k)$ is equal to the number of models of $T^{'-}_k$, where $T^{'-}_k$ is the reduction of all but Axioms (3b-e) of the theory $T'$ to the language $R_1,\cdots,R_k$ and $S_1,\cdots,S_k$.
    \item $L(k)$ satisfies the following recursive definition.
    Let
    \[J(0)=1 \text{ and } J(k+1)= 2(k+1)J(k)+\sum_{i=2}^{k+1} {k+1 \choose i} J(k+1-i).\]
    Then
    \[L(k)=\sum_{i=1}^{k} J(i) {k \choose i}.\]
    \item $L(k)$ are the coefficients in the exponential generating function
    \[H(x)=\frac{e^x}{2-x-e^x}.\]
    \item Let $W$ denote the product logarithm. Set $Z=2-W(e^2)\approx0.442854$ and $R=\frac{-e^2}{e^{W(e^2)}+e^2}\approx-0.6089389$.
    \[L(k)\sim A(k):=-k!R\Big(\frac1{Z}\Big)^{k+1}.\]
    \item In the limit, the proportion of homogeneous colored linear orderings that use all $k$ colors is $\frac1{W(2)}\approx0.6422007.$
    \end{enumerate}
\end{theorem}

\begin{proof}
    Note that (1) follows by the analog of Theorem \ref{thm:homogChar} for colored linear orderings (see \cite{ST08} Lemma 3.1) along with the same argument as in Proposition \ref{prop:spImage}.
    Therefore, it is enough to count the number of models of $T^{'-}_k$.

    We interpret $J(k)$ as the number of models of $T^{'-}_k$ where all $k$ colors are used (i.e., $\A\models \bigwedge_{i<k} \exists x ~R_i(x) \lor S_i(x)$).
    We demonstrate that $J(k)$ satisfies the given recurrence.
    $J(0)=1$ is plainly true, as only the empty linear ordering has no element of any color.
    We now show that 
    \[ J(k+1)= 2(k+1)J(k)+\sum_{i=2}^{k+1} {k+1 \choose i} J(k+1-i).\]
    Note that $\A\models T^{'-}_{k+1}$ must have finitely many elements; no color can be reused, and every element takes at least one color.
    Let $x\in\A$ be the first element, and let $|x|$ be the number of colors applied to $x$.
    Note that $\A-\{x\}$ uses $|x|$ less colors that $\A$.
    Up to relabeling the colors, $\A-\{x\}$ is a model of $T^{'-}_{k+1-|x|}$.
    There are exactly $J(k+1-|x|)$ such models.
    If $|x|=1$, either $R_j$ or $S_j$ may hold of $x$ and it may have one of $k+1$ colors, so there are $2(k+1)J(k)$ models of this form.
    If $|x|=i>1$, then some collection of $i$ many $S_j$ hold of $x$, so there are ${k+1 \choose i}$ models of this form.
    Summing these terms gives the desired recurrence.

    Now we can define $I(k)$ in terms of $J(k)$.
    Say $\A\models T^{'-}_k$ uses $i$ colors.
    There are ${k \choose i}$ sets of $i$ colors to choose from.
    Each of these sets has $J(i)$ many models.
    Thus, there are $J(i) {k \choose i}$ models of $T'_k$ using $i$ colors.
    It follows immediately that 
    \[L(k)=\sum_{i=1}^{k} J(i) {k \choose i}.\]

    We now provide the closed-form solution for $L(k)$ described in (3).
    Let $g_1(x)=e^{x}+x-1$, the exponential generating function counting pure non-empty sets where we allow for two different sets of size 1.
    Let $g_2=\frac1{1-x}$, the exponential generating function counting linear orderings of a finite set.
    Models $\A\models T^{'-}_k$ where all $k$ colors are used can be thought of as a composite $g_2\circ g_1$ structure.
    In particular, the elements of the partition correspond to elements of $\A$.
    These elements are linearly ordered (giving them a global $g_2$ structure), and to each of them we assign a pure set of colors, where single colors can be assigned in two different ways (giving each element of the partition a $g_1$ structure).
    The functional composition principal for exponential generating functions (see \cite[Chapter 2]{AnCom}) gives that if $f(x)$ is the exponential generating function for $J(k)$, then
    \[f(x)= g_2\circ g_1(x)=\frac{1}{1-(e^{x}+x-1)}=\frac{1}{2-x-e^x}. \]
    We now consider an arbitrary $\A\models T'_k$.
    Let $g_3=e^x$, the exponential generating function for counting pure sets.
    These models can be thought of as a $g_3\times f$ product structure.
    In particular, some arbitrary number of colors is first assigned to a set, and the remaining $\ell$ colors are used to build a model of $T'_\ell$.
    The multiplication principle for exponential generating functions (see \cite[Chapter 2]{AnCom}) gives that
    \[H(x)=g_3(x)\times f(x)=\frac{e^x}{2-x-e^x},\]
    as desired.

    We now conduct the asymptotic analysis from (4).
    This follows from the general method of Coefficient Asymptotics for meromorphic functions (see, for example, \cite[Chapter VI.5.2]{AnCom}).
    Specifically, note that $H(x)$ is a meromorphic function with simple poles where $2-e^x-x=0$.
    The solutions for this equation are given by $x=2-W_k(e^2)$, and it has a unique real valued solution given by $x=2-W_0(e^2)= 2-W(e^2)=Z$.
    As $H(x)$ has positive Taylor coefficients, by Pringsheim’s Theorem (\cite{AnCom} Theorem IV.6), $Z$ is the dominant pole for $H(x)$.
    An explicit, standard calculation gives that $R=\frac{-e^2}{e^{W(e^2)}+e^2}$ is the residue at this pole.
    In particular, at $Z$, $H(x)\sim R(x-Z)^{-1}$.
    In coefficients, we obtain that
    \[[x^k]H(x)\sim[x^k]R(x-Z)^{-1} \implies \frac{L(k)}{k!}\sim R\Big(\frac1{Z}\Big)^{k+1},\]
    so $L(k)\sim A(k)$ as desired.

    We lastly prove the claim from (5).
    We saw above that $J(k)$ counts the number of homogeneous colored linear orderings that use all $k$ colors.
    Furthermore,
    \[f(x)=\frac1{2-x-e^x}\]
    is an exponential generating function for $J(k)$.
    Crucially, $f(x)$ has the same poles as $H(x)$, so $Z$ is the dominant pole for $f(x)$ as well.
    The residue at $Z$ for $f(x)$ can be explicitly calculated to be $S=\frac{-e^{W(e^2)}}{e^{W(e^2)}+e^2}$.
    Analogous to the above, we have that $f(x)\sim S(1-Z)^{-1}$ so
    \[[x^k]f(x)\sim[x^k]S(1-Z)^{-1} \implies \frac{J(k)}{k!}\sim S\Big(\frac1{Z}\Big)^{k+1}.\]
    This immediately yields that,
    \[\frac{J(k)}{I(k)}\sim \frac{S}{R}=\frac{e^{W(e^2)}}{e^2}=\frac1{W(e^2)}\approx0.6422007,\]
    as desired.
\end{proof}

The first few terms for $L(k)$ are given as follows: 1, 3, 14, 95, 858, 9687, 131244, 2074515, 37475342, 761600375,17197534296, 427167206259, 11574924994554,...
The first few terms for $A(k)$ are approximately given by 1.37496, 3.10493, 14.0224, 94.9907, 857.986,...
According to numerical observations, $\frac{L(k)}{A(k)}$ converges to $1$ quite quickly.
For example, after ten terms, $\frac{L(k)}{A(k)}$ agrees with $1$ up to ten decimal places.
This follows the general principle that the error for this sort of approximation ought to be at most exponential (\cite{AnCom} Theorem IV.10) while our sequences of interest grow far more quickly than that rate.

Lastly, we note that, in the limit, there are very few sp-homogeneous linear orderings when compared to homogeneous colored linear orderings. 

\begin{proposition}
    $I(k)$ is $O(k!2.123^k)$. In particular, $\lim_{k\to\infty}\frac{I(k)}{L(k)}=0$.
\end{proposition}

\begin{proof}
    Let $J(k)$ be the function that counts the models of the following theory, $T^*_k$.
    Let $T^*_k$ be the theory that is defined by the following axioms in the language $(<,\{S_i\}_{i\in k})$. All of the $S_i$ are unary predicates.
    \begin{enumerate}
        \item $<$ is a linear ordering.
        \item $\forall x ~ \bigvee_{i\in k} S_i(x)$
        \item For each $i\in k$, $\forall x,y ~ \big(S_i(x)\land S_i(y)\big)\to x=y$.
    \end{enumerate}
    Consider a $C_{n,m}$ homogeneous linear ordering $L$ with $n+m+1=k$.
    The multicolored ordering $F(L)$ is a model of $T'$.
    $T^*_k$ is constructed explicitly so that its models are exactly the $S_i$ points in such an $F(L)$.
    In particular, $F(L)$ is determined by
    \begin{enumerate}
        \item a model of $T^*_k$, $\M$;
        \item a finite $k$-colored linear ordering with no repeated colors, $\N$;
        \item the assurance that the set of labels used in $\M$ and $\N$ do not intersect;
        \item an assignment $a:|M|+|N|\to 2$ such that $a^{-1}(0)$ has size $|M|$, $a^{-1}(1)$ has size $|N|$ and no two adjacent elements are assigned $1$.
    \end{enumerate} 
    To be explicit, the points $x\in F(L)$ with some $R_i(x)$ assigned are the points with $a(x)=1$, and the points with some $S_i(x)$ assigned are the points with $a(x)=0$.
    The two models in (1) and (2) are constructed by restricting attention to the $a(x)=0$ and $a(x)=1$ points, respectively.

    We now estimate $J(k)$.
    Models of $T^*_k$ are exactly linearly ordered non-empty sets.
    The size of non-empty sets is exponentially generated by $f(x)=e^x-1$ while the size of linear orderings is exponentially generated by $g(x)=\frac1{1-x}$.
    The functional composition principal for exponential generating functions (see \cite[Chapter 2]{AnCom}) yields that if $h(x)$ is the exponential generating function for $J(k)$ then
    \[h(x)=g\circ f(x)=\frac{1}{2-e^x}.\]
    The general method of Coefficient Asymptotics for meromorphic functions applies.
    $h(x)$ has simple poles where $e^x=2$.
    As $h(x)$ has positive, real Taylor coefficients, Pringsheim's Theorem (\cite{AnCom} Theorem VI.6) states that the smallest real pole, $ln(2)$, is the dominant pole for $h(x)$.
    In particular $h(x)\sim R(x-ln(2))^{-1}$ where $R$ is the residue for $h(x)$ at $x=ln(2)$.
    In coefficients,
    \[[x^k]h(x)\sim[x^k]R(x-ln(2))^{-1} \implies \frac{J(k)}{k!}\sim R\Big(\frac1{ln(2)}\Big)^{k+1}.\]
    For this calculation, it is enough to note that this yields that $J(k)$ is $O\Big(k!\big(\frac1{ln(2)}\big)^k\Big)$.

    Note that the number of uniquely labeled linear orderings using all $k$ colors is $k!$.
    Moreover, the number of length $r$ sequences of zeroes and ones with $s$ ones and no two adjacent ones is exactly ${r-s+1 \choose s}$.
    Let $\hat{I}(m)$ denote the number of models of $T'$ using exactly the first $m$ labels.
    The decomposition of a model of $T'$ outlined above gives that
    \[ \hat{I}(m)\leq \sum_{k=0}^{\lceil \frac{m}2\rceil} k!J(m-k){m \choose k}{m-k+1 \choose k}.\]
    To be more explicit,
    \begin{itemize}
        \item $k$ is the number of ones in $a$;
        \item the ${m \choose k}$ decides which labels will be used for the $a^{-1}(0)$ points versus the $a^{-1}(1)$ points;
        \item $k!$ counts the number of labeled linear orderings that can be placed on the $a^{-1}(1)$;
        \item $J(m-k)$ gives the number of models of $T^*_{m-k}$;
        \item ${m-k+1 \choose k}$ gives an upper bound for the number of ways to give an assignment $a$ for fixed $k$-uniquely labeled linear ordering and fixed model of $T^*_{m-k}$ (it is only an upper bound because the size of the model of $T^*_{m-k}$ may be smaller than $m-k$).
    \end{itemize} 
    We now asymptotically analyze 
    \[\sum_{k=0}^{\lceil \frac{m}2\rceil} k!J(m-k){m \choose k}{m-k+1 \choose k}=O\Bigg(\sum_{k=0}^{\lceil \frac{m}2\rceil}k!\Big(\frac1{ln(2)}\Big)^{m-k}(m-k)!{m \choose k}{m-k+1 \choose k}\Bigg).\]
    Let us focus on the quantity inside the summand
    \[k!\big(\frac1{ln(2)}\big)^{m-k}(m-k)!{m \choose k}{m-k+1 \choose k}=\Big(\frac1{ln(2)}\Big)^{m-k}m!{m-k+1 \choose k}.\]
    As the $m!$ is constant among choices of $k$, we focus on
    \[R(m,k)=\Big(\frac1{ln(2)}\Big)^{m-k}{m-k+1 \choose k}.\]
    Let us parameterize $k=pm$ so that $m-k=(1-p)m$.
    This gives that,
    \[R(m,p)=\Big(\frac1{ln(2)}\Big)^{(1-p)m}{(1-p)m+1 \choose pm}. \]
    Using the approximation for Binomial coefficients where the terms are linearly related, in this case by a limiting factor of $\frac{p}{1-p}$ (see e.g. \cite{asymptopia} Chapter 5.3), we get that
    \[R(m,p)=O\Bigg(\Big(\frac1{ln(2)}\Big)^{(1-p)m}2^{m(1-p)H(\frac{p}{1-p})}\Bigg)=O\Bigg(\Big(\Big(\frac1{ln(2)}\Big)^{(1-p)}2^{(1-p)H(\frac{p}{1-p})}\Big)^m\Bigg). \]
    Where $H(x)=-x\log_2(x)-(1-x)\log_2(1-x)$.
    We now give an upper bound for
    \[\Big(\frac1{ln(2)}\Big)^{(1-p)}2^{(1-p)H(\frac{p}{1-p})}.\]
    We accomplish this by maximizing this function.
    This single variable optimization problem can be explicitly (though admittedly exhaustively) solved using elementary methods of calculus.
    The maximum is at $p=\frac{-1 + \sqrt{1 + 4 \ln(2)}}{2 \sqrt{1 + 4\ln(2)}}$ yielding an explicit maximum approximated by $2.12243$.
    \begin{figure}
        \centering
        \includegraphics[width=0.5\linewidth]{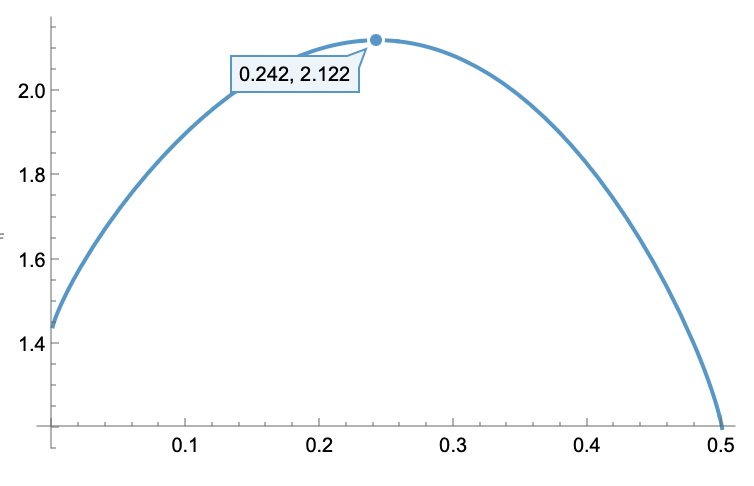}
        \caption{A graph for $\Big(\frac1{ln(2)}\Big)^{(1-p)}2^{(1-p)H(\frac{p}{1-p})}$ from $0$ to $\frac12$ with special attention on its maximum.}
        \label{fig:maxfunction}
    \end{figure}
    This can be confirmed by any sufficiently powerful computer algebra system and is visually represented in Figure \ref{fig:maxfunction}.
    This gives that
    \[R(m,p)=O(2.123^m)\]
    regardless of the value of $p$ and therefore the value of $k$.
    This gives that
    \[\hat{I}(k)=O\Bigg(m!\sum_{k=0}^{\lceil \frac{m}2\rceil}2.123^m\Bigg)=O(m!2.123^m).\]
    Finally, this gives that
    \[I(k)=\sum_{m=1}^k {k \choose m}\hat{I}(k) = O\Bigg(\sum_{m=1}^k {k \choose m}m!2.123^m \Bigg)=O\Bigg(k!\sum_{m=1}^k \frac{1}{(k-m)!}2.123^m \Bigg)=O(k!2.123^k),\]
    as desired.
    To obtain the desired limiting relation, note that $-\frac1Z$ from the previous theorem is about $2.258$, so the exponential factor for $L(k)$ is greater than that for $I(K)$.
\end{proof}

It is not known if the above bound is the best possible.
Given the techniques involved in proving it, this seems unlikely.
We conclude with an open question along these lines.

\begin{question}
    Does $I(k)=O(k!2^k)$?
\end{question}

\bibliographystyle{amsplain}
\bibliography{dclo}

\end{document}